\documentclass{amsart}
\usepackage[utf8]{inputenc}
\usepackage{amsmath,amsthm,amssymb,url}
\usepackage[margin = 2.5cm]{geometry}
\usepackage{stmaryrd}
\usepackage{tikz}
\newtheorem{theorem}{Theorem}

\newtheorem{prop}[theorem]{Proposition}
\newtheorem{cor}[theorem]{Corollary}
\newtheorem{lemma}[theorem]{Lemma}
\newtheorem{conj}{Conjecture}
\newtheorem{question}{Question}
\theoremstyle{definition}
\newtheorem{alg}{Algorithm}
\newtheorem*{example}{Example}
\newtheorem*{remark}{Remark}

\newcommand{\ttt}{{\tilde{t}}}
\newcommand{\cat}{c}  
\newcommand{\p}[1]{\mathcal #1}
\newcommand{\q}{q}   
\newcommand{\s}{{\mathfrak S}}
\newcommand{\x}{{\bf x}}
\DeclareMathOperator{\id}{id}
\newcommand{\Z}{{\mathbb Z}}

\title{
On the enumeration of tanglegrams and tangled chains}

\author{Sara C.  Billey}%
\address{Department of Mathematics, University of Washington,
Seattle, WA 98195, USA}
\email{}%
\urladdr{http://www.math.washington.edu/~billey/}
\thanks{The first author was partially supported by the National Science Foundation grant DMS-1101017. The second author was supported by Research Program Z1-5434 and Research Project BI-US/14-15-026 of the Slovenian Research Agency. The third author was supported by National Science Foundation grant DMS-1223057.}

\author{Matja\v z Konvalinka}%
\address{Department of Mathematics,
University of Ljubljana,
Jadranska 21, Ljubljana, Slovenia}
\email{}%
\urladdr{http://www.fmf.uni-lj.si/~konvalinka/}
\thanks{}

\author{Frederick A. Matsen IV}%
\address{Computational Biology Program,
Fred Hutchinson Cancer Research Center,
Seattle, WA 98109, USA}
\email{}%
\urladdr{http://matsen.fredhutch.org/}
\thanks{}

\date{\today}

\begin{document}

\begin{abstract}
  Tanglegrams are a special class of graphs appearing in applications
concerning cospeciation and coevolution in biology and computer
science.  They are formed by identifying the leaves of two rooted
binary trees.  We give an explicit formula to count the number of
distinct binary rooted tanglegrams with $n$ matched
vertices, along with a simple asymptotic formula and an algorithm for
choosing a tanglegram uniformly at random.  The enumeration formula is
then extended to count the number of tangled chains of binary trees of
any length.  This includes a new formula for the number of binary
trees with $n$ leaves.  We also give a conjecture for the expected
number of cherries in a large randomly chosen binary tree and an
extension of this conjecture to other types of trees.
\end{abstract}

\maketitle

\section{Introduction}

Tanglegrams are graphs obtained by taking two binary rooted trees with
the same number of leaves and matching each leaf from the tree on the
left with a unique leaf from the tree on the right.  This construction
is used in the study of cospeciation and coevolution in biology.  For
example, the tree on the left may represent the phylogeny of a host,
such as gopher, while the tree on the right may represent a parasite,
such as louse \cite{Hafner1988-da}, \cite[page 71]{Page.2002}.  One important problem is to reconstruct
the historical associations between the phylogenies of host and
parasite under a model of parasites switching hosts, which is an
instance of the more general problem of \emph{cophylogeny estimation}.
See \cite{Page.2002,Pevzner.Shamir,Scornavacca:2011} for applications in biology.
Diaconis and Holmes have previously demonstrated how one can encode a phylogenetic
tree as a series of binary matchings \cite{Diaconis1998-gn}, which is a distinct
use of matchings from that discussed here.

In computer science, the Tanglegram Layout Problem (TL) is to find a
drawing of a tanglegram in the plane with the left and right trees
both given as planar embeddings with the smallest number of crossings
among (straight) edges matching the leaves of the left tree and the
right tree \cite{bukin.etal.2008}. These authors point out that tanglegrams occur in the analysis of software projects and clustering problems.

In this paper, we give the exact enumeration of tanglegrams with $n$
matched pairs of vertices, along with a simple asymptotic formula and an
algorithm for choosing a tanglegram uniformly at random.  We refer to
the number of matched vertices in a tanglegram as its \emph{size}.
Furthermore, two tanglegrams are considered to be
equivalent if one is obtained from the other by replacing the tree on
the left or the tree on the right by isomorphic trees.  For example,
in Figure~\ref{fig:t3}, the two  non-equivalent tanglegrams of size 3 are
shown.

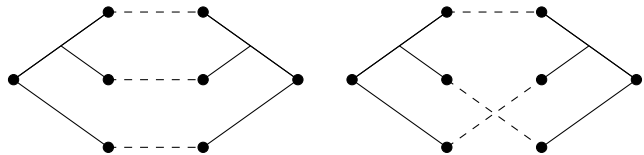
\begin{figure}[h!]
\begin{center}
\begin{tikzpicture}[scale = .9]
\newcommand{\treea}[2]{\coordinate (v1) at (#1,0); \coordinate (v2) at (#1+#2,0.5);\coordinate (v3) at (#1+2*#2,1);\coordinate (v4) at (#1+2*#2,0);\coordinate (v5) at (#1+2*#2,-1);\draw[fill] (v1) circle (.5ex);\draw[fill] (v3) circle (.5ex);\draw[fill] (v4) circle (.5ex);\draw[fill] (v5) circle (.5ex);\draw (v1) -- (v3);\draw (v1) -- (v3);\draw (v2) -- (v4);\draw (v1) -- (v5);}
\newcommand{\tanglea}[5]{\treea{#1}{#2} \treea{#1+6*#2}{-#2} \draw[dashed] (#1 + 2*#2,1) -- (#1 + 4*#2,2 - #3); \draw[dashed] (#1 + 2*#2,0) -- (#1 + 4*#2,2 - #4); \draw[dashed] (#1 + 2*#2,-1) -- (#1 + 4*#2,2 - #5);}
\tanglea 0 {0.7} 1 2 3
\tanglea {5} {0.7} 1 3 2
\end{tikzpicture}
\end{center}
\caption{The tanglegrams of size $3$.}\label{fig:t3}
\end{figure}

We state our main results here postponing some definitions until Section~\ref{sec:background}. The following is our main theorem.
\begin{theorem}\label{thm:1}
 The number of tanglegrams of size $n$ is
 $$t_{n}= \sum_\lambda \frac{\prod_{i=2}^{\ell(\lambda)} \left(2(\lambda_i+\cdots+\lambda_{\ell(\lambda)})-1\right)^2}{z_\lambda},$$
 where the sum is over binary partitions of $n$ and $z_\lambda$ is defined by Equation \eqref{z}.
\end{theorem}
The first 10
terms of the sequence $t_n$ starting at $n=1$ are
$$
1,1,2,13,114,1509,25595,535753,13305590,382728552,
$$
see
\cite[A258620]{oeis} for more terms.

\begin{example}
 The binary partitions of $n=4$ are $(4)$, $(2,2)$, $(2,1,1)$ and $(1,1,1,1)$, so
 $$t_4 = \frac{1}{4} + \frac{3^2}{8} + \frac{3^2 \cdot 1^2}{4} + \frac{5^2 \cdot 3^2 \cdot 1^2}{24} = 13$$
as shown in Figure~\ref{fig:t4}.
It takes a computer only a moment to compute
 $$t_{42} = 33889136420378480492869677415186948305278176263020722832251621520063757$$
 and under a minute to compute all 3160 integer digits of $t_{1000}$ using a recurrence based on Theorem~\ref{thm:1} given in Section~\ref{sec:recurrence}.
\end{example}

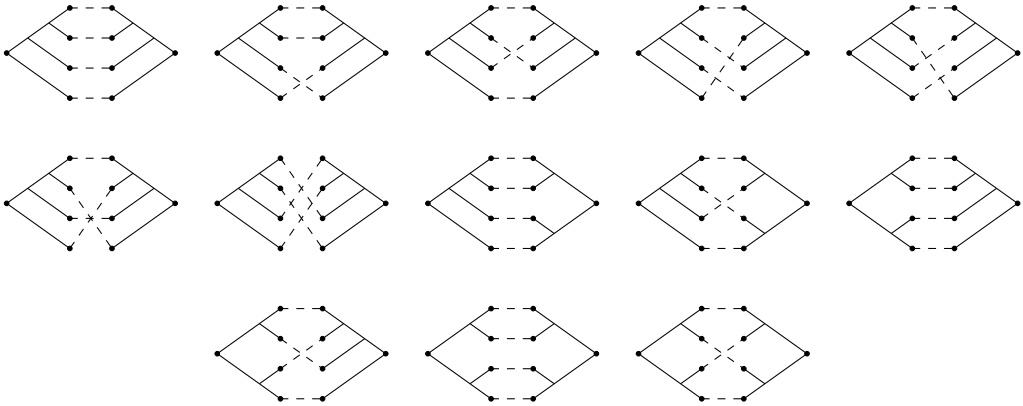
\begin{figure}[h!]
\begin{center}
\begin{tikzpicture}[scale = 0.4]
\newcommand{\treeb}[3]{\coordinate (v1) at (#1,#3); \coordinate (v2) at (#1+#2,#3+0.5);\coordinate (v3) at (#1+2*#2,#3+1);\coordinate (v4) at (#1+3*#2,#3+1.5);\coordinate (v5) at (#1+3*#2,#3+0.5);\coordinate (v6) at (#1+3*#2,#3-0.5);\coordinate (v7) at (#1+3*#2,#3-1.5);\draw[fill] (v1) circle (.5ex);\draw[fill] (v4) circle (.5ex);\draw[fill] (v5) circle (.5ex);\draw[fill] (v6) circle (.5ex);\draw[fill] (v7) circle (.5ex);\draw (v1) -- (v4);\draw (v3) -- (v5);\draw (v2) -- (v6);\draw (v1) -- (v7);}
\newcommand{\treec}[3]{\coordinate (v1) at (#1,#3); \coordinate (v2) at (#1+2*#2,#3-1);\coordinate (v3) at (#1+2*#2,#3+1);\coordinate (v4) at (#1+3*#2,#3+1.5);\coordinate (v5) at (#1+3*#2,#3+0.5);\coordinate (v6) at (#1+3*#2,#3-0.5);\coordinate (v7) at (#1+3*#2,#3-1.5);\draw[fill] (v1) circle (.5ex);\draw[fill] (v4) circle (.5ex);\draw[fill] (v5) circle (.5ex);\draw[fill] (v6) circle (.5ex);\draw[fill] (v7) circle (.5ex);\draw (v1) -- (v4);\draw (v3) -- (v5);\draw (v2) -- (v6);\draw (v1) -- (v7);}
\newcommand{\tangleb}[7]{\treeb{#1}{#2}{#7} \treeb{#1+8*#2}{-#2}{#7} \draw[dashed] (#1 + 3*#2,#7+1.5) -- (#1 + 5*#2,#7+2.5-#3); \draw[dashed] (#1 + 3*#2,#7+0.5) -- (#1 + 5*#2,#7+2.5-#4); \draw[dashed] (#1 + 3*#2,#7-0.5) -- (#1 + 5*#2,#7+2.5-#5); \draw[dashed] (#1 + 3*#2,#7-1.5) -- (#1 + 5*#2,#7+2.5-#6);}
\newcommand{\tanglec}[7]{\treeb{#1}{#2}{#7} \treec{#1+8*#2}{-#2}{#7} \draw[dashed] (#1 + 3*#2,#7+1.5) -- (#1 + 5*#2,#7+2.5-#3); \draw[dashed] (#1 + 3*#2,#7+0.5) -- (#1 + 5*#2,#7+2.5-#4); \draw[dashed] (#1 + 3*#2,#7-0.5) -- (#1 + 5*#2,#7+2.5-#5); \draw[dashed] (#1 + 3*#2,#7-1.5) -- (#1 + 5*#2,#7+2.5-#6);}
\newcommand{\tangled}[7]{\treec{#1}{#2}{#7} \treeb{#1+8*#2}{-#2}{#7} \draw[dashed] (#1 + 3*#2,#7+1.5) -- (#1 + 5*#2,#7+2.5-#3); \draw[dashed] (#1 + 3*#2,#7+0.5) -- (#1 + 5*#2,#7+2.5-#4); \draw[dashed] (#1 + 3*#2,#7-0.5) -- (#1 + 5*#2,#7+2.5-#5); \draw[dashed] (#1 + 3*#2,#7-1.5) -- (#1 + 5*#2,#7+2.5-#6);}
\newcommand{\tanglee}[7]{\treec{#1}{#2}{#7} \treec{#1+8*#2}{-#2}{#7} \draw[dashed] (#1 + 3*#2,#7+1.5) -- (#1 + 5*#2,#7+2.5-#3); \draw[dashed] (#1 + 3*#2,#7+0.5) -- (#1 + 5*#2,#7+2.5-#4); \draw[dashed] (#1 + 3*#2,#7-0.5) -- (#1 + 5*#2,#7+2.5-#5); \draw[dashed] (#1 + 3*#2,#7-1.5) -- (#1 + 5*#2,#7+2.5-#6);}
\tangleb {0} {0.7} 1 2 3 4 {0}
\tangleb {7} {0.7} 1 2 4 3 {0}
\tangleb {14} {0.7} 1 3 2 4 {0}
\tangleb {21} {0.7} 1 3 4 2 {0}
\tangleb {28} {0.7} 1 4 2 3 {0}
\tangleb {0} {0.7} 1 4 3 2 {-5}
\tangleb {7} {0.7} 3 4 1 2 {-5}
\tanglec {14} {0.7} 1 2 3 4 {-5}
\tanglec {21} {0.7} 1 3 2 4 {-5}
\tangled {28} {0.7} 1 2 3 4 {-5}
\tangled {7} {0.7} 1 3 2 4 {-10}
\tanglee {14} {0.7} 1 2 3 4 {-10}
\tanglee {21} {0.7} 1 3 2 4 {-10}
\end{tikzpicture}
\end{center}
\caption{The $13$ tanglegrams of size $4$.}\label{fig:t4}
\end{figure}
\medskip

We use the main theorem to study the asymptotics
of the sequence $t_n$. It turns out that
$$\frac{t_n}{n!} \sim \frac{e^{\frac 1 8} 4^{n-1}}{\pi n^3},$$
see Corollary \ref{asymptotic} for an explanation and better estimates.

\medskip

A side result of the proof is a new formula for the number of
inequivalent binary trees, called the Wedderburn-Etherington numbers \cite[A001190]{oeis}.

\begin{theorem} \label{thm:trees}
 The number of inequivalent binary trees with $n$ leaves is
 $$b_n = \sum_\lambda \frac{\prod_{i=2}^{\ell(\lambda)} (2(\lambda_i+\cdots+\lambda_{\ell(\lambda)})-1)}{z_\lambda},$$
 where the sum is over binary partitions of $n$.
\end{theorem}

A \emph{tangled chain} is an ordered sequence of $k$ binary trees with matchings between neighboring trees in the sequence. For $k = 1$, these are inequivalent binary trees, and for $k = 2$, these are tanglegrams, so the following generalizes Theorems~\ref{thm:1} and \ref{thm:trees}.

\begin{figure}[h!]
\begin{center}
\begin{tikzpicture}[scale = 0.45]
\newcommand{\treea}[3]{\coordinate (v1) at (#3+0,#1); \coordinate (v2) at (#3+0.5,#1+#2);\coordinate (v3) at (#3+1,#1+2*#2);\coordinate (v4) at (#3+0,#1+2*#2);\coordinate (v5) at (#3+-1,#1+2*#2);\draw[fill] (v1) circle (.5ex);\draw[fill] (v3) circle (.5ex);\draw[fill] (v4) circle (.5ex);\draw[fill] (v5) circle (.5ex);\draw (v1) -- (v3);\draw (v1) -- (v3);\draw (v2) -- (v4);\draw (v1) -- (v5);}
\newcommand{\tanglef}[8]{\treea #2 {-0.7} {#1} \treea {#2-4} {0.7} {#1+2} \treea {#2} {-0.7} {#1+4} \draw[dashed] (#1+1,#2-2*0.7) -- (#1+2+1-#3+1,#2-4+2*0.7); \draw[dashed] (#1,#2-2*0.7) -- (#1+2+1-#4+1,#2-4+2*0.7); \draw[dashed] (#1-1,#2-2*0.7) -- (#1+2+1-#5+1,#2-4+2*0.7); \draw[dashed] (#1+2+1,#2-4+2*0.7) -- (#1+4+2-#6,#2-2*0.7); \draw[dashed] (#1+1+1,#2-4+2*0.7) -- (#1+4+2-#7,#2-2*0.7); \draw[dashed] (#1+1,#2-4+2*0.7) -- (#1+4+2-#8,#2-2*0.7);}
\tanglef 0 0 1 2 3 1 2 3
\tanglef 7 0 1 2 3 1 3 2
\tanglef {14} 0 1 3 2 1 2 3
\tanglef {21} 0 1 3 2 1 3 2
\tanglef {28} 0 1 3 2 3 1 2
\end{tikzpicture}
\end{center}
\caption{The tangled chains of length $3$ for $n = 3$.}\label{fig:tchains3}
\end{figure}
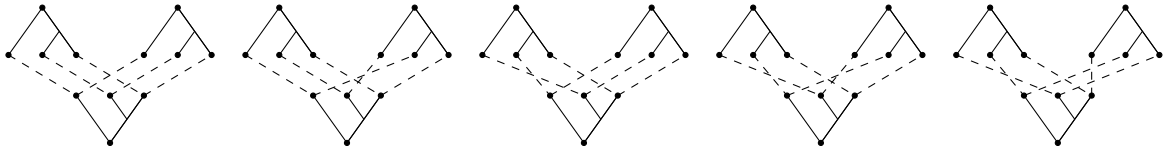

In terms of computational biology, tangled chains of length $k$ formalize the essential input to a variety of problems on $k$ leaf-labeled (phylogenetic) trees (e.g.~\cite{Whidden2014-yt}).

\begin{theorem}
 The number of ordered tangled chains of length $k$ for $n$ is
 $$\sum_\lambda \frac{\prod_{i=2}^{\ell(\lambda)} \left(2(\lambda_i+\cdots+\lambda_{\ell(\lambda)})-1\right)^k}{z_\lambda},$$
 where the sum is over binary partitions of $n$.
 \label{thm:tchains}
\end{theorem}

\begin{example}
 For $n = k = 3$, we have partitions $(2,1)$ and $(1,1,1)$, and the theorem gives
 $$\frac{1^3}{2} + \frac{3^3 \cdot 1^3}{6} = 5,$$ as shown in Figure
 \ref{fig:tchains3}.  For $k=3$, the number of tangled chains on trees
 with $n$ leaves gives rise to the sequence starting
$$
 1,1,5,151,9944,1196991,226435150,61992679960,23198439767669,11380100883484302.
$$
See \cite[A258486]{oeis} for more terms.
\end{example}

From the enumerative point of view, it is also quite natural to ask
how likely a particular tree $T$ is to appear on one side or the other
of a uniformly selected tanglegram.  In Section~\ref{sec:remarks}, we give a simple
explicit conjecture for the asymptotic growth of the expected number
of copies of $T$ on one side of a tanglegram as a function of $T$ and
the size of the tanglegram.  For example, the cherries of a binary
tree are pairs of leaves connected by a common parent.
We conjecture that the expected number of cherries in one of the binary trees
of a tanglegram of size $n$ chosen in the uniform distribution is $n/4$.

Further discussion of the applications of tanglegrams along with
several variations on the theme are described in \cite{arniePaper}.
In particular, tanglegrams can be used to compute the
subtree-prune-regraft distance between two binary trees.

The paper proceeds as follows.  In Section~\ref{sec:background}, we
define our terminology and state the main theorems.  We prove the main
theorems in Section~\ref{proof}. Section~\ref{sec:random} contains an
algorithm to choose a tanglegram uniformly at random for a given
$n$. In Section~\ref{sec:asymptotics}, we give several asymptotic
approximations to the number of tanglegrams with increasing accuracy
and complexity. In Section~\ref{sec:recurrence}, we give a recursive
formula for both the number of tanglegrams and for tangled chains. We
conclude with several open problems and conjectures in
Section~\ref{sec:remarks}.

\section{Background}\label{sec:background}

In this section, we recall some vocabulary and notation on partitions and
trees.  This terminology can also be found in standard textbooks
on combinatorics such as \cite{ec1}.  We use these terms to give the
formal definition of tanglegrams and the notation used in the main
theorems.

A \emph{partition} $\lambda=(\lambda_1,\lambda_2,\ldots, \lambda_k)$
is a weakly decreasing sequence of positive integers.  The length
$\ell(\lambda)$ of a partition is the number of entries in the
sequence, and $|\lambda|$ denotes the sum of the entries of $\lambda$.
We say $\lambda$ is a \emph{binary partition} if all its parts are
equal to a nonnegative power of $2$.  Binary partitions have
appeared in a variety of contexts, see for instance
in \cite{Knuth.1966,Konvalinka.Pak.2014,Sloane.Sellers} and
\cite[A000123]{oeis}. When writing partitions, we sometimes omit parentheses and commas.

If $\lambda$ is a nonempty binary partition with $m_i$
occurrences of the letter $2^i$ for each $i$, we also denote $\lambda$ by
$(1^{m_0},2^{m_1},4^{m_2},8^{m_3}, \ldots, (2^j)^{m_j})$ where $2^j=\lambda_1$ is the maximum value
in $\lambda$.  Given $\lambda = (1^{m_0},2^{m_1}, \ldots, (2^j)^{m_j})$, let
$z_\lambda$ denote the product
\begin{equation} \label{z}
z_\lambda = 1^{m_0}2^{m_1} \cdots (2^j)^{m_j}m_0! m_1! m_2! \cdots m_j!.
\end{equation}
The numbers $z_\lambda$ are well known since the number of
permutations in $\s_n$ with cycle type $\lambda$ is $n!/z_\lambda$
\cite[Prop. 1.3.2]{ec1}. For example, for $\lambda = 44211 = (1^2,2^1,4^2)$, $z_\lambda = 1^2 \cdot 2^1 \cdot 4^2 \cdot 2! \cdot 1! \cdot 2! = 128$.

A rooted tree has one distinguished vertex assumed to be a common
ancestor of all other vertices.  The neighbors of the root are its
\emph{children}.  Each vertex other than the root has a unique
parent going along the path back to the root, the other neighbors are
its children.  In a binary tree, each vertex either has two children or
no children.  A vertex with no children is a \emph{leaf}, and a vertex
with two children is an \emph{internal vertex}.  Two binary rooted
trees with labeled leaves are said to be \emph{equivalent} if there
is an isomorphism from one to the other as graphs mapping the root of
one to the root of the other.  
Let $B_n$ be the set of inequivalent binary rooted
trees with $n\geq 1$ leaves, and let $b_n$ be the number of elements
in the set $B_n$.  The sequence of $b_n$'s for $n\geq 1$ begins
$$
1, 1, 1, 2, 3, 6, 11, 23, 46, 98.
$$

We can inductively define a linear order on rooted trees as follows. We say that $T > S$ if either:
\begin{itemize}
 \item $T$ has more leaves than $S$
 \item $T$ and $S$ have the same number of leaves, $T$ has subtrees $T_1$ and $T_2$, $T_1 \geq T_2$, $S$ has subtrees $S_1$ and $S_2$, $S_1 \geq S_2$, and $T_1 > S_1$ or $T_1 = S_1$ and $T_2 > S_2$
\end{itemize}
We assume that every tree $T$ in $B_n$, $n \geq 2$, is presented so that $T_1 \geq T_2$, where $T_1$ is the left subtree (or upper subtree if the tree is drawn with the root on the left or on the right) and $T_2$ is the right (or lower) subtree.

For each tree $T \in B_n$, we can identify its automorphism group
$A(T)$ as follows.  Fix a labeling on the leaves of $T$ using the
numbers $1,2,\ldots, n$.  Label each internal vertex by the union of
the labels for each of its children.  The edges in $T$ are pairs of
subsets from $[n]:=\{1,\ldots, n\}$, each representing the label of a
child and its parent.  Let $v=[v(1),v(2),\ldots, v(n)]$ be a
permutation in the symmetric group $\s_n$.  Then, $v \in A(T)$ if
permuting the leaf labels by the function $i \mapsto v(i)$ for each
$i$ leaves the set of edges fixed.

A theorem due to Jordan \cite{Jordan} tells us that if $T$ is a
tree with subtrees $T_1$ and $T_2$, then $A(T)$ is isomorphic to
$A(T_1) \times A(T_2)$ if $T_1 \neq T_2$, and to the wreath product
$A(T_1) \wr \Z_2$ if $T_1 = T_2$. Since the automorphism group
of a tree on one vertex is trivial, this implies that the general
$A(T)$ can be obtained from copies of $\Z_2$ by direct and wreath
products (see \cite{arniePaper} for more details).
Furthermore, if $T_1 \neq T_2$, then the conjugacy type
of an element of $A(T)$ is $\lambda^1 \cup \lambda^2$, where $\lambda^i$
is the conjugacy type of an element of $A(T_i)$, $i = 1,2$, and
$\lambda^{1} \cup \lambda^{2}$ is the multiset union of the two sequences
written in decreasing order. If $T_1 = T_2$, then for an arbitrary
element of $A(T)$ either the leaves in each subtree remain in that subtree,
or all leaves are mapped to the other subtree. The conjugacy type
of an element of $A(T)$ is then either $\lambda^1 \cup \lambda^2$, where $\lambda^i$
is the conjugacy type of an element of $A(T_i)$, $i = 1,2$, or it
is $2 \lambda^1$,  where $\lambda^1$ is the conjugacy type of an element of $A(T_1)$.
In particular, the conjugacy type of any element of the automorphism
group of a binary tree must be a binary partition.

Next, we define tanglegrams.   Given a permutation $v \in \s_n$ along
with two trees $T,S \in B_n$ each with leaves labeled $1,\ldots, n$,
we construct an \emph{ordered binary rooted tanglegram} $(T,v,S)$
of size $n$ with $T$ as the left tree, $S$ as the right tree, by
identifying leaf $i$ in $T$ with leaf $v(i)$ in $S$.  Note, $(T,v,S)$
and $(T',v',S')$ are considered to represent the same tanglegram
provided $T = T'$,\ $S = S'$ as trees and $v'=uvw$ where $u 
\in A(T)$ and $w \in A(S)$.  Let $T_n$ be the set of all ordered binary rooted
tanglegrams of size $n$, and let $t_n$ be the number of elements in
the set $T_n$.  For example, $t_3=2$ and $t_4=13$.
Figures~\ref{fig:t3} and~\ref{fig:t4} show the tanglegrams of sizes 3
and 4 where we draw the leaves of the left and right tree on separate
vertical lines and show the matching using dashed lines.  The dashed
lines are not technically part of the graph, but this visualization
allows us to give a planar drawing of the two trees.

We remark that the \emph{planar binary trees} with $n\geq 2$ leaves are a
different family of objects from $B_n$ that also come up in this paper.  These are trees embedded in
the plane so the left child of a vertex is distinguishable from the
right child.  The planar binary trees with $n+1$ leaves are well known
to be counted by Catalan numbers
$$\cat_{n} = \frac{1}{n+1}\binom{2n}{n} =\frac{2^n(2n-1)!!}{(n+1)!}$$
because they clearly satisfy the Catalan recurrence
$$
\cat_n = \cat_0\cat_{n-1} + \cat_1\cat_{n-2} +  \cat_2\cat_{n-3} + \cdots + \cat_{n-1}\cat_{0}
$$
with $\cat_0=\cat_1=1$.   For example, there are
$c_2=2$ distinct planar binary trees with $3$ leaves which are mirror
images of each other while $b_3=1$.  The sequence of $c_n$'s for
$n\geq 0$ begins
$$
1, 1, 2, 5, 14, 42, 132, 429, 1430, 4862 ,
$$
see \cite[A000108]{oeis}.

Dulucq and Guibert \cite{Dulucq.Guibert} have studied ``twin binary
trees'', which are pairs of planar binary trees with matched vertices.
This is the planar version of tanglegrams.   They show that twin
binary trees are in bijection with Baxter permutations.  The Baxter
permutations in $S_n$ are enumerated by a formula due to
Chung-Graham-Hoggart-Kleiman \cite{Chung.Graham.Hoggatt.Kleiman}
$$
a_n = \frac{\sum_{k=1}^{n}\binom{n+1}{k-1} \binom{n+1}{k} \binom{n+1}{k+1}}
{\binom{n+1}{1} \binom{n+2}{2}}
$$
See also the bijective proof by Viennot \cite{Viennot.1981}, and
further refinements \cite {Cori.Dulucq.Viennot,MR1417289}.

\section{Proof of the main theorems} \label{proof}

The focus of this section is the proof of Theorem \ref{thm:1}, namely that
$$t_{n}= \sum_\lambda \frac{\prod_{i=2}^{\ell(\lambda)} \left(2(\lambda_i+\cdots+\lambda_{\ell(\lambda)})-1\right)^2}{z_\lambda},$$
where the sum is over binary partitions of $n$. The proof of Theorem \ref{thm:1} reflects the chronological steps of discovery.
Theorem \ref{thm:trees} will follow from a auxiliary result, and the proof of Theorem~\ref{thm:tchains} is similar and is included at the end of the section.

The number of tanglegrams is, by definition, equal to
$$t_n = \sum_T \sum_S |\p C(T,S)|,$$ where the sums on the right are
over inequivalent binary trees with $n$ leaves, and $\p C(T,S)$ is
the set of double cosets of the symmetric group $\s_n$ with respect to the double
action of $A(T)$ on the left and $A(S)$ on the right. Let us fix $T\in B_n$
and $S\in B_n$ and write $\p C = \p C(T,S)$. Then
$$|\p C| = \sum_{C \in \p C} 1 = \sum_{C \in \p C} \frac{|C|}{|C|} = \sum_{C \in \p C} \sum_{w \in  C} \frac{1}{|C|} = \sum_{w \in \s_n} \frac{1}{|C_w|},$$
where $C_w$ is the double coset of $\s_n$ that contains $w$. It is known (e.g.\ \cite[Theorem 2.5.1 on page 45 and Exercise 40 on page 49]{Herstein})
that the size of the double coset $C_w = A(T)wA(S)$ is the quotient
$$\frac{|A(T)| \cdot |A(S)|}{|A(T) \cap wA(S)w^{-1}|},$$
and therefore,
$$|\p C| = \sum_{w \in \s_n} \frac{|A(T) \cap wA(S)w^{-1}|}{|A(T)| \cdot |A(S)|}.$$
We have
$$\sum_{w \in \s_n} |A(T) \cap wA(S)w^{-1}| = \sum_{w \in \s_n} \sum_{u \in A(T)} \sum_{v \in A(S)} \llbracket u = wvw^{-1} \rrbracket = \sum_{u \in A(T)} \sum_{v \in A(S)} \sum_{w \in \s_n}  \llbracket u = w v w^{-1}\rrbracket,$$
where $\llbracket \cdot \rrbracket$ is the indicator function. Now $u = w v w^{-1}$ can only be true if $u$ and $v$ are permutations of the same conjugacy type $\lambda$, which must necessarily be a binary partition as noted above. Furthermore, if $u$ and $v$ are both of type $\lambda$, then there are $z_\lambda$ permutations $w$ for which $u = w v w^{-1}$. That means that
\begin{equation} \label{cosetsize}
|\p C(T,S)| = \frac{\sum_\lambda |A(T)_\lambda| \cdot |A(S)_\lambda| \cdot z_\lambda}{|A(T)| \cdot |A(S)|},
\end{equation}
where $A(T)_\lambda$ (respectively, $A(S)_\lambda$) denotes the elements of $A(T)$ (resp., $A(S)$) of type $\lambda$.

Equation \eqref{cosetsize} is already quite useful for computing all tanglegrams with fixed left and right trees. For example, if $T$ and $S$ are both the least symmetric tree with only one cherry, then $A(T) = A(S) = \{\id,(1,2)\}$, the sum is over only two binary partitions of size $n$, namely $(1,\ldots, 1)$ and $(2,1,\ldots, 1)$, and we get
$$|\p C| = \frac{n! + 2 (n-2)!}{2 \cdot 2} = \frac{(n^2-n+2) (n-2)!}{4}.$$
In some other cases the summation is over many more $\lambda$'s, and can get quite complicated.

However, to get the formula for $t_n$ we want to sum Equation~\eqref{cosetsize} over all pairs of trees, and fortunately a change of the order of summation helps. Indeed, we have
\begin{align}\label{eqn:tn.form.1}
  t_n &= \sum_T \sum_S \frac{\sum_\lambda |A(T)_\lambda| \cdot |A(S)_\lambda| \cdot z_\lambda}{|A(T)| \cdot |A(S)|} = \sum_\lambda z_\lambda \cdot \sum_T \sum_S \frac{|A(T)_\lambda| \cdot |A(S)_\lambda|}{|A(T)| \cdot |A(S)|}
  \\&= \sum_\lambda z_\lambda \cdot   \left( \sum_T \frac{|A(T)_\lambda|}{|A(T)|}\right)^2 ,
   \end{align}
and the main theorem will be proved once we have shown the following proposition.

\begin{prop} \label{prop}
 For a binary partition $\lambda$,
$$
   \sum_{T \in B_n} \frac{|A(T)_\lambda|}{|A(T)|} = \frac{\prod_{i=2}^{\ell(\lambda)} (2\left(\lambda_i+\cdots+\lambda_{\ell(\lambda)}\right)-1)}{z_\lambda},
$$
 where  $A(T)_\lambda$ denotes the elements of $A(T)$ of type $\lambda$.
\end{prop}

The proposition also implies Theorem~\ref{thm:trees}, as
$$\sum_T 1 = \sum_T \sum_\lambda \frac{|A(T)_\lambda|}{|A(T)|} = \sum_\lambda \sum_T \frac{|A(T)_\lambda|}{|A(T)|}.$$

If $\lambda = 1^n$, then $|A(T)_\lambda| = 1$ for all $T \in B_n$, so the proposition is saying that
$$\sum_T \frac 1{|A(T)|} = \frac{(2n-3)!!}{n!} = \frac{\cat_{n-1}}{2^{n-1}}.$$
This is equivalent to $\sum_T 2^{n-1}/|A(T)| = \cat_{n-1}$.
Since $2^{n-1}/|A(T)|$ counts all planar binary trees isomorphic to $T$, this is just the
well-known fact that there are $\cat_{n-1}$ planar binary trees with $n$ leaves.

For a general $\lambda$, however, the proposition is far from obvious. What we need is a recursion satisfied by the expression on the right, analogous to the recursion $\cat_{n} = \cat_0 \cat_{n-1} + \cat_1 \cat_{n-1} + \cdots + \cat_{n-1} \cat_0$ for Catalan numbers.

\begin{lemma} \label{lemma}
  For a nonempty subset $S = \{i_1<i_2<\ldots<i_k\}$ of the natural numbers define
  \begin{equation}\label{eq:rec.for.r}
r_S(x_1,x_2,\ldots) = (x_{i_2} + \cdots + x_{i_k} - 1)(x_{i_3} + \cdots + x_{i_k} - 1) \cdots (x_{i_{k-1}}+x_{i_k} - 1) (x_{i_k} - 1).
\end{equation}
 Let $n \geq 2$, let $\x$ denote variables $x_1,x_2,\ldots$, and let $\x/2$ denote $x_1/2,x_2/2,\ldots$. Then
 $$r_{[n]}(\x) = 2^{n-1} r_{[n]}(\x/2) + \sum_{1 \in S \subsetneq [n]} r_S(\x) \cdot r_{[n] \setminus S}(\x).$$
\end{lemma}

\begin{example}
 For $n = 3$, the lemma says that
 $$(x_2 + x_3 - 1)(x_3 - 1) = (x_2 + x_3 - 2)(x_3 - 2) + 1 \cdot (x_3-1) + (x_2-1) \cdot 1 + (x_3-1) \cdot 1,$$
 where the last three terms on the right-hand side correspond to subsets $\{1\}$, $\{1,2\}$, and $\{1,3\}$, respectively.
 As another example, take $x_i = 2$ for all $i$. Then $r_S(\x) = (2|S|-3)!!$ (where we interpret $(-1)!!$ as $1$), $r_S(\x/2) = 0$,  and by the obvious symmetry of $S$ and $[n] \setminus S$ the lemma yields
 $$2 \cdot (2n-3)!! = \sum_{k = 1}^{n-1} \binom n k (2k-3)!! (2n-2k-3)!!,$$
 which is equivalent to the standard recurrence for Catalan numbers.
\end{example}

\bigskip

\begin{proof}[Proof of Lemma \ref{lemma}]
  The proof is by induction on $n$. For $n = 2$, the statement is simply $x_2 - 1 = (x_2 - 2) + 1 \cdot 1$. Assume that the statement holds for $n-1$, and let us prove it for $n$. Both sides are linear functions in $x_2$, so it is sufficient to prove that they have the same coefficient at $x_2$ and that they give the same result for one value of $x_2$.

 The coefficient of $x_2$ in $r_{[n]}(\x)$ (resp., $2^{n-1}r_{[n]}(\x/2)$) is clearly $r_{[2,n]}(\x)$ (resp., $2^{n-2} r_{[2,n]}(\x/2)$).
 On the other hand, $r_S(\x) \cdot r_{[n] \setminus S}(\x)$ contains $x_2$ if and only if $2 \in S$, in which case the coefficient at $x_2$ is $r_{S \setminus \{1\}}(\x) \cdot r_{[2,n] \setminus S}(\x)$. The coefficients on both sides are equal by induction.

 Plug the value $x_2 = 2 - x_3 - \cdots - x_n$ into both sides. Clearly, the left-hand side becomes $r_{[n]\setminus \{2\}}(\x)$. It is easy to see that if $2 \in S$, then $r_S(\x) \cdot r_{[n] \setminus S}(\x) + r_{S \setminus \{2\}}(\x) \cdot r_{([n] \setminus S) \cup \{2\}}(\x) = 0$. That means that all the terms in the summation cancel out except $r_{[n]\setminus \{2\}}(\x) \cdot r_{\{2\}}(\x) = r_{[n]\setminus \{2\}}(\x)$. Obviously, $r_{[n]}(\x/2) = 0$, so the right-hand side also equals $r_{[n]\setminus \{2\}}(\x)$.
\end{proof}

\begin{proof}[Proof of Proposition \ref{prop}]
Say $\lambda$ is a binary partition of $n$.  The proof is by
induction on $n$.  For $n = 1$, the statement is obvious.  Assume that the
statement holds for all binary partitions up to size $n-1$. Our task is to show
$$
\sum_T
\frac{|A(T)_\lambda|}{|A(T)|} = \frac{r_{[\ell(\lambda)]}(2\lambda_1, 2\lambda_2, 2\lambda_3,\ldots)}{z_{\lambda}}
$$
by showing the left hand side satisfies a recurrence similar to
\eqref{eq:rec.for.r}.

Given $T \in B_n$, let $T_1$ and $T_2$ be the subtrees of the root in
$T$. Fix a labeling on the leaves of $T$ such that the leaves of $T_1$ are
labeled $[1,k]$ and the leaves of $T_2$ are labeled $[k+1,n]$.
Consider each $A(T_i)$ to be a subgroup of the permutations of the
leaf labels for $T_i$.  We can obtain a permutation of type $\lambda$
in $A(T)$ in one of two ways.
First, we can choose permutations $w_1 \in A(T_1), w_2 \in A(T_2)$ of types $\lambda^1$ and $\lambda^2$, then $w_1 w_2$ is a permutation of $A(T)$ of type $\lambda$.
Second, if all parts of $\lambda$ are at least $2$ and $T_1 =
T_2$ (and in particular $n = 2k$), we can choose an arbitrary permutation $w_1 \in A(T_1)$ and
another permutation $w_2 \in A(T_1)$ specifically of type $\lambda/2:=
(\lambda_1/2, \lambda_2/2,\ldots)$ and construct a permutation $w\in
A(T)$ of cycle type $\lambda$ as follows.  Say $f:[1,k]\longrightarrow
[k+1,n]$ mapping $i$ to $i+k$  induces an isomorphism of $T_1$ and $T_2$.  Define the ``tree
flip permutation'' $\pi$ to be the product of the transpositions
interchanging $i$ with $f(i)$ for all $1\leq i \leq k$.
Now take the
product
$$
w = \pi w_1 \pi w_1^{-1} \pi w_2.
$$
It is clear that $w \in
A(T)$ since it is the product of permutations in $A(T)$.  Observe also
that the cycles of $w$ are constructed so the leaf labels of $T_1$
interleave the leaf labels of $T_2$ in the cycles of $w_2$ so $w$ will
have cycle type $\lambda$.  For example, if $\lambda =(6,4)$, then
$|\lambda|=10$ and $\pi=(1\ 6)(2\ 7)(3\ 8)(4\ 9) ( 5\ 10)$.  If we
choose $w_1=(1\ 4) (2\ 5) (3)$ and $w_2=(6\ 9\ 7) (8\ 10)$ then $w=\pi
w_1 \pi w_1^{-1} \pi w_2=(6\ 1\ 9\ 5\ 7\ 4) (8\ 2\ 10\ 3)$, all in
cycle notation. Also, every element of $A(T)$ is constructed in one of these
two ways.

We need to be careful to differentiate between the cases
when the subtrees $T_1,T_2$ are different and when they are equivalent.
We
have
 \begin{align*}
  \sum_T \frac{|A(T)_\lambda|}{|A(T)|} & = \sum_{T_1 > T_2} \frac{|A(T)_\lambda|}{|A(T)|} + \sum_{T_1 = T_2} \frac{|A(T)_\lambda|}{|A(T)|} =\\ \sum_{T_1 > T_2} \left(\sum_{\lambda^1 \cup \lambda^2 = \lambda} \frac{|A(T_1)_{\lambda^1}| \cdot |A(T_2)_{\lambda^2}|}{|A(T_1)| \cdot |A(T_2)|} \right)
  &+ \sum_{T_1} \frac{\left(\sum_{\lambda^1 \cup \lambda^2 = \lambda} |A(T_1)_{\lambda^1}| \cdot |A(T_1)_{\lambda^2}|\right) + |A(T_1)| \cdot |A(T_1)_{\lambda/2}|}{2|A(T_1)|^2}
 \end{align*}
 or equivalently
 \begin{equation}\label{eq:p.rec}
2\sum_{T \in B_{n}} \frac{|A(T)_\lambda|}{|A(T)|} = \sum_{T_1 \in B_{n/2}} \frac{|A(T_1)_{\lambda/2}|}{|A(T_1)|} +
\sum_{\lambda^1 \cup \lambda^2 = \lambda} \left( \sum_{T_1 \in B_{|\lambda^{1}|}}\frac{|A(T_1)_{\lambda^1}|}{|A(T_1)|} \right)
\left( \sum_{T_2 \in B_{|\lambda^{2}|}} \frac{|A(T_2)_{\lambda^2}|}{|A(T_2)|} \right).
 \end{equation}
 Let
 $$\q_\lambda = \frac{\prod_{i=2}^{\ell(\lambda)} (2(\lambda_i+\cdots+\lambda_{\ell(\lambda)})-1)}{z_\lambda} =
\frac{r_{[\ell(\lambda)]}(2\lambda_1, 2\lambda_2, 2\lambda_3,\ldots)}{z_{\lambda}} ;$$
 the notation also makes sense if $\lambda_{\ell(\lambda)} = 1/2$, as in that case $\q_\lambda = 0$.
 By the induction hypothesis and \eqref{eq:p.rec}, it suffices to prove that
 \begin{equation}\label{pind}
  2 \q_\lambda = \q_{\lambda/2} + \sum_{\lambda^1 \cup \lambda^2 = \lambda} \q_{\lambda^1} \cdot \q_{\lambda^2}.
 \end{equation}
 After multiplying both sides by $z_\lambda$, this is
 \begin{align*}
 2& \prod_{i=2}^{\ell(\lambda)} (2(\lambda_i+\cdots+\lambda_{\ell(\lambda)})-1) = 2^{\ell(\lambda)}\prod_{i=2}^{\ell(\lambda)} (\lambda_i+\cdots+\lambda_{\ell(\lambda)}-1) \\
 +& \sum_{\lambda^1 \cup \lambda^2 = \lambda} \binom{\lambda}{\lambda^1,\lambda^2} \cdot \prod_{i=2}^{\ell(\lambda^1)} (2(\lambda^1_i+\cdots+\lambda^1_{\ell(\lambda^1)})-1)\cdot \prod_{i=2}^{\ell(\lambda^2)} (2(\lambda^2_i+\cdots+\lambda^2_{\ell(\lambda^2)})-1),
 \end{align*}
 where $\binom{\lambda}{\lambda^1,\lambda^2} = \prod_i \binom{m_i(\lambda)}{m_i(\lambda^1)}$. This equality holds by Lemma \ref{lemma} with $x_i = 2\lambda_i$.
\end{proof}

We conclude this section with the proof of Theorem~\ref{thm:tchains}.

\newcommand{\mct}{C^{\mathbf{T}}}  
\newcommand{\mctw}{C^{\mathbf{T}}(w_{1},\ldots, w_{k-1})}  
\newcommand{\mctwp}{C^{\mathbf{T'}}(w_{2},\ldots, w_{k-1})}  

\begin{proof}[Proof of Theorem~\ref{thm:tchains}]
Let $\mathbf{T}=(T_{1},T_{2},\ldots, T_{k})$ be an ordered list of binary trees in
$B_{n}$. Define $\mct$ to be the set of
``multicosets'' of $\s_{n}$ with respect to  $A(T_{1})\times
A(T_{2})\times \cdots \times A(T_{k})$.  More concretely, given
$(w_{1},\ldots, w_{k-1}),(w'_{1},\ldots, w'_{k-1}) \in (\s_{n})^{k-1}$,
we say $(w_{1},\ldots, w_{k-1}) \equiv_{\mathbf{T}} (w'_{1},\ldots, w'_{k-1})$
provided there exist $t_{i} \in A(T_{i})$ such that
$w_{i}=t_{i}w'_{i}t_{i+1}$ for all $i=1,\ldots, k-1$.  Then,
$\mct$ is the set of equivalence classes
modulo $\equiv_{\mathbf{T}}$.  By definition, the number of tangled chains of
length $k$ and size $n$, denoted $t(k,n)$, is given by
\begin{equation}\label{eqn:k.1}
t(k,n) = \sum |\mct|
\end{equation}
where the sum is over all ordered lists
$\mathbf{T}=(T_{1},T_{2},\ldots, T_{k})$ of trees $T_i \in B_n$.

Fix a particular list of trees $\mathbf{T}=(T_{1},T_{2},\ldots, T_{k})$, and
let $\mctw$ be the multicoset in $\mct$
containing $(w_{1},\ldots, w_{k-1})$.  Clearly,
\begin{equation*}
|\mct| = \sum_{w_{1} \in \s_{n}} \sum_{w_{2} \in \s_{n}} \cdots \sum_{w_{k-1} \in \s_{n}} \frac{1}{|\mctw|}.
\end{equation*}
We give a recurrence for $|\mctw|$ in terms of the following subgroup.
Let $A(\mctw)$ be the subgroup of all $ t_1 \in A(T_1)$ such that
there exist $t_{i} \in A(T_{i})$ for $2\leq i\leq k$ satisfying
$w_{i}=t_{i}w_{i}t_{i+1}$ for all $i=1,\ldots, k-1$.  In this case,
$(t_1 w_{1},w_2,\ldots, w_{k-1}) \equiv_{\mathbf{T}}
(w_{1},w_2,\ldots, w_{k-1})$ so we think of $A(\mctw)$ as the ``left
automorphism group'' of $\mctw$.   Observe that
\[
A(\mctw) = A(T_{1}) \cap w_{1} A(T_{2})w_{1}^{-1} \cap \cdots \cap w_{1}w_{2}\cdots w_{k-1} A(T_{k}) w_{k-1}^{-1} \cdots w_{2}^{-1}w_{1}^{-1},
\]
so
\[
|A(\mctw)| = \sum_{i=1}^k \ \sum_{t_i \in A(T_i)}  \llbracket t_1 = w_1 t_2 w_1^{-1}\rrbracket \cdot \llbracket t_2 = w_2 t_3 w_2^{-1}\rrbracket \cdots \llbracket t_{k-1} = w_{k-1} t_k w_{k-1}^{-1}\rrbracket.
\]
Now let $\mathbf{T'}=(T_{2},\ldots, T_{k})$.
  For each $(v_{2},\ldots,v_{k-1}) \in \mctwp$, we can
prepend a $v_{1}$ to create a distinct element
$(v_1, v_{2},\ldots,v_{k-1}) \in \mctw$
exactly when $v_{1}$ is in  $A(T_1) w_{1} A(\mctwp)$ which is again a double coset of $\s_n$.
Thus, by the formula for double cosets we have
\begin{align*} 
|\mctw| &= \frac{|A(T_{1})| \cdot
|A(\mctwp)|}{|A(\mctw)|}
\cdot |\mctwp|
\\
&=\frac{|A(T_{1})| \cdot |A(T_{2})| \cdots |A(T_{k})|}
{|A(\mctw)|}
\end{align*}
by induction on $k$.
Therefore,
\begin{equation} \label{eqn:k.4}
|\mct| = \sum_{w_{1} \in \s_{n}} \sum_{w_{2}
\in \s_{n}} \cdots \sum_{w_{k-1} \in \s_{n}}
\frac{|A(\mctw)|}{|A(T_{1})| \cdot |A(T_{2})| \cdots |A(T_{k})|},
\end{equation}
where the denominators do not depend on the $w_i$'s.

Focusing on the sum in the numerator in \eqref{eqn:k.4}, we have
\begin{align*}
  \sum_{(w_1,w_2,\ldots, w_{k-1})}  &|A(\mctw)|  \\
  &= \sum_{(w_1,w_2,\ldots, w_{k-1})}\ \sum_{t_1 \in A(T_1)} \cdots  \sum_{t_k \in A(T_k)} \llbracket t_1 = w_1 t_2 w_1^{-1}\rrbracket  \cdots \llbracket t_{k-1} = w_{k-1} t_k w_{k-1}^{-1}\rrbracket
  \\
  &=  \sum_{t_1 \in A(T_1)} \cdots  \sum_{t_k \in A(T_k)} \sum_{(w_1,w_2,\ldots, w_{k-1})} \ \llbracket t_1 = w_1 t_2 w_1^{-1}\rrbracket \cdots \llbracket t_{k-1} = w_{k-1} t_k w_{k-1}^{-1}\rrbracket
  \end{align*}
and so with similar logic as before, noting that the summand will be nonzero exactly when $t_1$, $t_2, \ldots, t_k$ are all of the same conjugacy type $\lambda$,
\begin{equation} \label{cosetsize.2}
|\mct| = \frac{\sum_\lambda |A(T_1)_\lambda| \cdot |A(T_2)_\lambda| \cdots |A(T_k)_\lambda| \cdot z_\lambda^{k-1}}{|A(T_1)| \cdot |A(T_2)| \cdots |A(T_k)|}.
\end{equation}
Plugging \eqref{cosetsize.2} into \eqref{eqn:k.1}, we obtain
\begin{align*}
t(k,n) &= \sum_{(T_1,\ldots, T_k)} \frac{\sum_\lambda
  |A(T_1)_\lambda| \cdot |A(T_2)_\lambda| \cdots |A(T_k)_\lambda| \cdot
  z_\lambda^{k-1}}{|A(T_1)| \cdot |A(T_2)| \cdots |A(T_k)|}
\\
&= \sum_\lambda
z_\lambda^{k-1} \cdot \left( \sum_{T \in B_n} \frac{|A(T)_\lambda|}{|A(T)|}\right)^k,
\end{align*}
and Theorem~\ref{thm:tchains}  now follows from Proposition~\ref{prop}.
\end{proof}

\section{Random generation of tanglegrams and inequivalent binary trees}\label{sec:random}

In this section, we describe an algorithm in 3 stages to produce a
random tanglegram in $T_n$.  The stages are based on
Equation~\eqref{eqn:tn.form.1} and the proof of Proposition~\ref{prop}.
A similar algorithm is also described to choose a random binary tree with $n$
leaves.
In this section, ``random'' will mean uniformly at random unless specified otherwise.

Recall from Section~\ref{proof} that if $T$ is a tree with equivalent
left and right subtrees, we denote by $\pi$ the ``tree flip
permutation'' between the subtrees.  Also, for a partition $\lambda$,
we defined
$$\q_\lambda = \frac{\prod_{i=2}^{\ell(\lambda)} (2(\lambda_i+\cdots+\lambda_{\ell(\lambda)})-1)}{z_\lambda}.$$
The $\q_\lambda$ notation also makes sense if $\lambda_{\ell(\lambda)} = 1/2$, as in that case $\q_\lambda = 0$.

\begin{alg}[Random generation of $w \in A(T)$] \label{alg1}
~\newline
\vspace{-.3cm}
~\newline
\emph{\textbf{Input:}} Binary tree $T \in B_n$.
  \\
\vspace{-.3cm}
  \\
 \emph{\textbf{Procedure:}} If $T$ is the tree with one vertex, let $w$ be the unique element of $A(T)$.
 Otherwise, the root of $T$ has subtrees $T_1$ and $T_2$.  Assume the leaves of $T_1$ are labeled $[1,k]$ and the leaves of $T_2$ are labeled $[k+1,n]$.  Use the algorithm recursively to produce $w_i \in A(T_i)$, $i=1,2$ where $A(T_1)$ is a subset of the permutations of $[1,n]$ which fix $[k+1,n]$ and $A(T_2)$ is a subset of the permutations of $[1,n]$ which fix $[1,k]$. Construct $w$ as follows.
\begin{itemize}
\item  If $T_1 \neq T_2$, set $w = w_1 w_2$.\\
\item If $T_1 = T_2$, choose either $w = w_1 w_2$ or 
  $w = \pi w_1 w_2$
  with equal probability.
\end{itemize}
\vspace{.3cm}
 \emph{\textbf{Output:}} Permutation $w \in A(T)$.
\end{alg}

\bigskip

\begin{alg}[Random generation of $T$ with non-empty $A(T)_\lambda$ and $w \in A(T)_\lambda$] \label{alg2}
~\newline
\vspace{-.3cm}
~\newline
  \emph{\textbf{Input:}} Binary partition $\lambda$ of $n$.
  \\
\vspace{-.3cm}
  \\
 \emph{\textbf{Procedure:}} If $n = 1$, let $T$ be the tree with one vertex, and let $w$ be the unique element of $A(T)$.
 \\
 Otherwise, pick a subdivision $(\lambda^1,\lambda^2)$ from
 $\{(\lambda^1,\lambda^2) \colon \lambda^1 \cup \lambda^2 = \lambda \}
 \cup \{(\lambda/2,\lambda/2) \}$, where $(\lambda^1,\lambda^2)$ is
 chosen with probability proportional to $\q_{\lambda^1}\q_{\lambda^2}$
 and $(\lambda/2,\lambda/2)$ with probability proportional to
 $\q_{\lambda/2}$.
 \begin{itemize}
 \item  If $\lambda^1,\lambda^2 \neq \lambda/2$, use the algorithm recursively to produce trees $T_1,T_2$ and permutations $w_1 \in A(T_1)_{\lambda^1}$, $w_2 \in A(T_2)_{\lambda^2}$. If necessary, switch $T_1 \leftrightarrow T_2$, $w_1 \leftrightarrow w_2$ so that $T_1 \geq T_2$. Let $T = (T_1,T_2)$, $w = w_1w_2$.
   \\
 \item  If $\lambda^1 = \lambda^2 = \lambda/2$, use the algorithm recursively to produce a tree $T_1$ and a permutation $w_2 \in A(T_1)_{\lambda/2}$, and use Algorithm \ref{alg1} to produce a permutation $w_1 \in A(T_1)$. Let $T = (T_1,T_1)$ and $w = \pi w_1 \pi w_1^{-1} \pi w_2$.
\end{itemize}
\vspace{.3cm}
 \emph{\textbf{Output:}} Binary tree $T$ and permutation $w \in A(T)_\lambda$.
\end{alg}

\bigskip

\begin{alg}[Random generation of tanglegrams] \label{alg3}
 ~\newline
 \vspace{-.3cm}
 ~\newline
  \emph{\textbf{Input:}} Integer $n$.
  \\
\vspace{-.3cm}
  \\
  \emph{\textbf{Procedure:}} Pick a random binary partition $\lambda$ of $n$ with probability proportional to $z_\lambda \q_\lambda^2$
  where $t_n = \sum z_\lambda \q_\lambda^2$.
  Use Algorithm \ref{alg2} twice to produce random trees $T$ and $S$ and permutations $u \in A(T)_\lambda$, $v \in A(S)_\lambda$.   Among the permutations $w$ for which $u = w v w^{-1}$, pick one at random
  from the $z_\lambda$ possibilities.
  \\
\vspace{-.3cm}
  \\
 \emph{\textbf{Output:}}  Binary trees $T$ and $S$ and double coset $A(T)wA(S)$, or equivalently $(T,w,S)$.
\end{alg}

\bigskip

\begin{alg}[Random generation of $T \in B_n$] \label{alg4}
~\newline
\vspace{-.3cm}
~\newline
\emph{\textbf{Input:}} Integer $n$.
  \\
\vspace{-.3cm}
\\
  \emph{\textbf{Procedure:}} Pick a random binary partition $\lambda$
  of $n$ with probability proportional to $\q_\lambda$.  Use Algorithm
  \ref{alg2} to produce a random tree $T$ (and a permutation $u \in
  A(T)_\lambda$).
\\
\vspace{-.3cm}
\\
  \emph{\textbf{Output:}} Binary tree $T$.
\end{alg}

Algorithm~\ref{alg4} is not the first of its kind, see also \cite{Furnas1984-yz}.

\bigskip

\begin{alg}[Random generation of tangled chains] \label{alg5}
 ~\newline
 \vspace{-.3cm}
 ~\newline
  \emph{\textbf{Input:}} Positive integers $k$ and $n$.
  \\
\vspace{-.3cm}
  \\
  \emph{\textbf{Procedure:}} Pick a random binary partition $\lambda$ of $n$ with probability proportional to $z_\lambda^{k-1} \q_\lambda^{k}$
  where  $t(k,n) = \sum z_\lambda^{k-1} \q_\lambda^{k}$.
 Use Algorithm \ref{alg2} $k$ times to produce random trees $T_{i}$ and permutations $u_{i} \in A(T_{i})_\lambda$ for $i=1,\ldots ,k$.
Among the permutations $w_{i}$ for which $u_{i} = w_{i} u_{i+1} w_{i}^{-1}$, pick one uniformly at random for each $i=1,\ldots ,k-1$.
  \\
\vspace{-.3cm}
  \\
 \emph{\textbf{Output:}}  $(T_{1},\dots , T_{k})$ and $(w_{1},\ldots, w_{k-1})$.
\end{alg}

\bigskip

\begin{theorem} For  any positive integer $n$, the following hold.
  \begin{itemize}
\item  Algorithm \ref{alg1} produces every permutation $w \in A(T)$ with probability $\frac 1{|A(T)|}$.
\\
\item  Algorithm \ref{alg2} produces every pair $(T,w)$, where $w \in A(T)_\lambda$, with probability $\frac{1}{|A(T)| \cdot \q_\lambda}$.
  \\
\item Algorithm \ref{alg3} produces every tanglegram with probability $\frac 1{t_n}$.
  \\
\item  Algorithm \ref{alg4} produces every inequivalent binary tree with probability $\frac 1{b_n}$.
\\
\item  Algorithm \ref{alg5} produces every tangled chain of length $k$ of trees in $B_{n}$  with probability $\frac 1{t(k,n)}$.
  \end{itemize}
  \end{theorem}

\begin{proof}
 The first two proofs are by induction, with the case $n=1$ being
 obvious.  The induction for Algorithm \ref{alg1} is trivial.

For Algorithm \ref{alg2}, say that we are given a binary partition
$\lambda$, a tree $T$ with $n=|\lambda|$ leaves, and $w \in
A(T)_\lambda$.  We compute the probability that Algorithm \ref{alg2}
produces $T$ and $w$. Assume first that $T_1 > T_2$ are the subtrees
of $T$. In particular, that means that $w$ can be written uniquely as
$w_1 w_2$, where $w_1 \in A(T_1)$ and $w_2 \in A(T_2)$. Say that $w_i$
is of type $\lambda^i$; we must have $\lambda = \lambda^1 \cup
\lambda^2$. If $\lambda^1 \neq \lambda^2$, there are two ways in which
Algorithm \ref{alg2} can produce $(T,w)$: either we partition
$\lambda$ into $(\lambda^1,\lambda^2)$, and then the algorithm
produces $(T_1,w_1)$ and $(T_2,w_2)$, or we partition $\lambda$ into
$(\lambda^2,\lambda^1)$, then the algorithm produces $(T_2,w_2)$ and
$(T_1,w_1)$, and finally switches $T_1 \leftrightarrow T_2$, $w_1
\leftrightarrow w_2$.  Since $T_{1}$ and $T_{2}$ are chosen
independently, we can apply \eqref{pind} and induction to obtain the probability that $(T,w)$ is chosen, namely
 $$2 \cdot \frac{\q_{\lambda^1}\q_{\lambda^2}}{2\q_\lambda} \cdot
 \frac{1}{|A(T_1)| \cdot \q_{\lambda^1}} \cdot \frac{1}{|A(T_2)| \cdot
   \q_{\lambda^2}} = \frac 1{|A(T_1)|\cdot |A(T_2)| \cdot \q_\lambda} =
 \frac 1 {|A(T)| \cdot \q_\lambda}.
 $$
 If $\lambda^1 = \lambda^2$, but $T_{1} \neq T_{2}$, there
 are again two ways in which Algorithm \ref{alg2} can produce $(T,w)$:
 we must partition $\lambda$ into $(\lambda^1,\lambda^1)$, and then it
 can either produce $(T_1,w_1)$ and $(T_2,w_2)$ or $(T_2,w_2)$ and
 $(T_1,w_1)$; in the latter case it switches $T_1 \leftrightarrow
 T_2$, $w_1 \leftrightarrow w_2$. Similarly, the probability is  $\frac{1}{|A(T)| \cdot \q_\lambda}$.

 Now assume that $T_1 = T_2$.  Either $w$ can be written as
 $w_1 w_2$, where $w_1 \in A(T_1)_{\lambda^1}$ and $w_2 \in
 A(T_2)_{\lambda^2}$, or as $\pi w_2 \pi w_2^{-1} \pi w_1$, where $w_1
 \in A(T_1)_{\lambda/2}$ and $w_2 \in A(T_1)$. In the first case,
 $(T,w)$ is produced with probability
 $$\frac{\q_{\lambda^1}\q_{\lambda^2}}{2\q_\lambda} \cdot \frac{1}{|A(T_1)| \cdot \q_{\lambda^1}} \cdot \frac{1}{|A(T_1)| \cdot \q_{\lambda^2}} = \frac 1{2 \cdot |A(T_1)|^2 \cdot \q_\lambda} = \frac 1 {|A(T)| \cdot \q_\lambda}.$$
 In the second case, it is produced with probability
 $$\frac{\q_{\lambda/2}}{2\q_\lambda} \cdot \frac{1}{|A(T_1)| \cdot
   \q_{\lambda/2}} \cdot \frac{1}{|A(T_1)|} = \frac 1{2 \cdot
   |A(T_1)|^2 \cdot \q_\lambda} = \frac 1 {|A(T)| \cdot \q_\lambda}.$$
 This finishes the case for Algorithm~\ref{alg2}.

The proof of the statement for Algorithm \ref{alg3} is essentially
just a rewriting of the proof from Section \ref{proof}; we include it
for completeness. We are given $n$ and a tanglegram $(T,w,S)$ with $T$
and $S$ binary trees with $n$ leaves, $C=A(T)wA(S)$ the double coset containing
$w$ with respect to $A(T)$ and $A(S)$, and we want to prove that
$P(T,S,C)$, the probability that this triple is produced by Algorithm
\ref{alg3}, is $1/t_n$.

 We proved that $\sum z_\lambda \q_\lambda^2 = t_n$, so the probability of choosing a binary partition $\lambda$ is $z_\lambda \q_\lambda^2/t_n$. So we have
 $$P(T,S,C) = \sum_\lambda \frac{z_\lambda \q_\lambda^2}{t_n}
 P(T,S,C|\lambda),$$ where $P(T,S,C|\lambda)$ is the conditional
 probability that $(T,S,C)$ is produced if $\lambda$ is chosen. We can
 further condition the probability: $P(T,S,C|\lambda) = \sum
 P(T,S,C|u,v,T,S,\lambda) \cdot P(u,v,T,S|\lambda)$, where the sum is
 over $u \in A(T)_\lambda$, $v \in A(S)_\lambda$. Furthermore,
 $$P(T,S,C|u,v,T,S,\lambda) = P(C|u,v) \text{\ and \ } P(u,v,T,S|\lambda) =  P(T,u|\lambda) \cdot P(S,v|\lambda),$$
and so
 \begin{align*}
 P(T,S,C) &= \sum_\lambda \frac{z_\lambda \q_\lambda^2}{t_n} \sum_{u \in A(T)_\lambda} \sum_{v \in A(S)_\lambda} P(C|u,v) \cdot \frac{1}{|A(T)| \cdot \q_\lambda} \cdot \frac{1}{|A(S)| \cdot \q_\lambda}\\
 &= \frac 1{t_n} \cdot \sum_\lambda \frac{z_\lambda}{|A(T)| \cdot |A(S)|} \cdot {\sum_{u \in A(T)_\lambda} \sum_{v \in A(S)_\lambda} \frac{|C \cap B^{u,v}|}{|B^{u,v}|}},
 \end{align*}
 where $B^{u,v} = \{w \in \s_n \colon u = w v w^{-1}\}$. We know that $|B^{u,v}| = z_\lambda$, so
 \begin{align*}
 P(T,S,C) &= \frac 1{t_n} \cdot \sum_\lambda \frac{1}{|A(T)| \cdot |A(S)|} \sum_{u \in A(T)_\lambda} \sum_{v \in A(S)_\lambda} \sum_{w \in C} \llbracket u = wvw^{-1} \rrbracket \\
 &= \frac 1{t_n} \cdot  \sum_{w \in C} \sum_\lambda \frac{1}{|A(T)| \cdot |A(S)|} \sum_{u \in A(T)_\lambda} \sum_{v \in A(S)_\lambda} \llbracket u = wvw^{-1} \rrbracket \\
 &= \frac 1{t_n} \cdot  \sum_{w \in C} \sum_\lambda \frac{|A(T)_\lambda \cap wA(S)_\lambda w^{-1}|}{|A(T)| \cdot |A(S)|} = \frac 1{t_n} \cdot  \sum_{w \in C} \frac{|A(T) \cap wA(S) w^{-1}|}{|A(T)| \cdot |A(S)|} \\
 &= \frac 1{t_n} \cdot  \sum_{w \in C} \frac 1{|C_w|} = \frac 1{t_n}.
 \end{align*}

 Finally, let us prove the statement for Algorithm \ref{alg4}. We have
 $$P(T) = \sum_{\lambda} P(T|\lambda) \cdot P(\lambda) =
 \sum_{\lambda} \frac{|A(T)_\lambda|}{|A(T)| \cdot \q_\lambda}
 \cdot \frac{\q_\lambda}{b_n}= \frac 1{b_n} \cdot \frac{\sum_\lambda
   |A(T)_\lambda|}{|A(T)|}= \frac 1{b_n},
 $$
 which proves that
 Algorithm \ref{alg4} produces every inequivalent binary tree with the
 same probability.  The proof for  Algorithm \ref{alg5} is similar to Algorithms~\ref{alg3} and~\ref{alg4} so we omit the formal proof.
\end{proof}

\section{Asymptotic expansion of $t_n$}\label{sec:asymptotics}

In this section, we use Theorem~\ref{thm:1} to obtain another formula
for $t_n$ and several formulas to approximate $t_n$ for large $n$.

\begin{cor} We have
\begin{equation} \label{equiv}
t_n = \frac{\cat_{n-1}^2n!}{4^{n-1}} \sum_\mu \frac{n(n-1)\cdots(n-|\mu|+1)}{z_\mu \cdot \prod_{i=1}^{\ell(\mu)} \prod_{j=1}^{\mu_i-1} (2n-2(\mu_1+\cdots+\mu_{i-1})-2j-1)^2},
\end{equation}
where the sum is over binary partitions $\mu$  with all parts
equal to a positive power of $2$ and $|\mu|\leq n$ including the empty partition
in which case the summand is 1.
\end{cor}

\begin{proof}
Every binary partition $\lambda$ of size $n$ can be expressed as $\mu
1^{n - |\mu|}$, where all parts of $\mu$ are at least $2$. We have
$z_\lambda = z_\mu (n-|\mu|)!$ and
\begin{align*}
\prod_{i=2}^{\ell(\lambda)} \left(2(\lambda_i+\cdots+\lambda_{\ell(\lambda)})-1\right) =& \prod_{i=1}^{\ell(\lambda)-1} \left(2(n-\lambda_1-\cdots-\lambda_{i})-1\right) \\
  =&\prod_{i=1}^{\ell(\mu)-1} \left(2(n-\mu_1-\cdots-\mu_{i})-1\right) \cdot (2n-2|\mu|-1)!!\\
  =& \frac{(2n-3)!!}{\prod_{i=1}^{\ell(\mu)} \prod_{j=1}^{\mu_i-1} (2n-2(\mu_1+\cdots+\mu_{i-1})-2j-1)}.
\end{align*}

Since $(2n-3)!!/n! = \cat_{n-1}/2^{n-1}$,
\eqref{equiv} is an equivalent way to express
the number of tanglegrams.
\end{proof}

The first
 few terms of the sum corresponding to
 partitions $\emptyset$, $(2)$, $(4)$, $(2,2)$, $(4,2)$, $(2,2,2)$,
 $(8)$ are
\begin{align*}
1 &+ \frac{n(n-1)}{2(2n-3)^2} + \frac{n(n-1)(n-2)(n-3)}{4(2n-3)^2(2n-5)^2(2n-7)^2} + \frac{n(n-1)(n-2)(n-3)}{8(2n-3)^2(2n-7)^2} \\
& + \frac{n(n-1)(n-2)(n-3)(n-4)(n-5)}{8(2n-3)^2(2n-5)^2(2n-7)^2(2n-11)^2} + \frac{n(n-1)(n-2)(n-3)(n-4)(n-5)}{48(2n-3)^2(2n-7)^2(2n-11)^2} \\
&+ \frac{n(n-1)(n-2)(n-3)(n-4)(n-5)(n-6)(n-7)}{8(2n-3)^2(2n-5)^2(2n-7)^2(2n-9)^2(2n-11)^2(2n-13)^2(2n-15)^2}.
\end{align*}

\begin{cor} \label{asymptotic}
 We have
 $$\frac{t_n}{n!} \sim \frac{e^{\frac 1 8}\cat_{n-1}^2}{4^{n-1} } \sim \frac{e^{\frac 1 8} 4^{n-1}}{\pi n^3} \qquad \mbox{and} \qquad
t_n \sim \frac { 2^{2 n - \frac 3 2} \cdot n^{n - \frac 5 2}}{\sqrt \pi \cdot e^{n-\frac 1 8}}.$$
We can also compute approximations of higher degree. For example, we have
\begin{align*}
t_n &= \frac{e^{\frac 1 8}\cat_{n-1}^2 n!}{4^{n-1}} \cdot \left( 1 + \frac{1}{4 \: n} + \frac{137}{256 \: n^2} + \frac{1285}{1024 \: n^3} + \frac{456017}{131072 \: n^4} + \frac{6140329}{524288 \: n^5} + O\left(n^{-6}\right) \right) \\
&= \frac {2^{2 n - \frac 3 2} \cdot n^{n - \frac 5 2}}{\sqrt \pi \cdot e^{n-\frac 1 8}} \cdot \left( 1 + \frac{13}{12 \: n} + \frac{3089}{2304 \: n^2} + \frac{931423}{414720 \: n^3} + \frac{826301423}{159252480 \: n^4} + \frac{211060350013}{13377208320 \: n^5} +  O\left(n^{-6}\right) \right).
\end{align*}
\end{cor}

\bigskip

\begin{proof}[Sketch of proof]

The crucial observation is that
 $$\frac{n(n-1)\cdots(n-|\mu|+1)}{z_\mu \cdot \prod_{i=1}^{\ell(\mu)} \prod_{j=1}^{\mu_i-1} (2n-2(\mu_1+\cdots+\mu_{i-1})-2j-1)^2} \sim \frac{n^{|\mu|}}{z_\mu \cdot (2n)^{2(|\mu|-\ell(\mu))}} = \frac{1}{2^{2(|\mu|-\ell(\mu))}\cdot z_\mu \cdot n^{|\mu|-2\ell(\mu)}}.
 $$

So, to find an asymptotic approximation of order $O(n^{-2m})$ or
 $O(n^{-2m-1})$, we only have to consider partitions $\mu$
 with $|\mu|-2\ell(\mu) \leq 2m$ in Equation~\eqref{equiv}.
 For $m = 0$, we only consider partitions of the type $22\cdots 2$.
The contribution of $\mu = 2^k$ is $1/(2^{2k} 2^k k!)$, and the sum converges to $\sum_k \frac{1}{2^{3k}k!} = e^{\frac 1 8}$.

 Similarly, the coefficient of $n^{-1}$ can be obtained by considering
 the coefficient of $n^{-1}$ in each of these terms, and the higher terms by
 considering in turn partitions of type $42^k$, $4^22^k$, $4^32^k$,
 $82^k$, etc. The last expansion is obtained by considering the
 asymptotic expansions of $\cat_{n-1}$ and $n!$.
\end{proof}

\section{A recurrence for enumerating  tanglegrams and tangled chains}\label{sec:recurrence}

In this section, we give a recurrence for computing $t_{n}$.  Recall
that for each nonempty binary partition $\lambda$, we can construct
its \emph{multiplicity vector} $m^{\lambda}=(m_{0}, m_{1},
m_{2},m_{3}, \ldots)$ where $m_{i}$ is the number of times $2^{i}$
occurs in $\lambda$.  The map $\lambda \mapsto m^{\lambda}$ is a
bijection from binary partitions to vectors of nonnegative integers
with only finitely many nonzero entries.  The quantity $z_{\lambda}$
for a binary partition $\lambda$ is easily expressed in terms of the
multiplicities in $m^{\lambda}$ as
\[
z_{\lambda} =  \prod_{h\geq 0}2^{h\cdot m_{h}}\ m_{h}!
=  \prod_{\substack{h\geq 0\\m_h \neq 0}} \prod_{j=1}^{m_{h}}j\cdot 2^{h}
\]

We will use the functions
\begin{equation}\label{e:factor}
f^2(s) := (2s-1)^{2},
\end{equation}

\begin{equation}\label{e:contrib}
c(h,m,s) := \prod_{j=1}^{m} \frac{f^2(s + j\cdot 2^{h})}{j\cdot 2^{h}},
\end{equation}
and
\begin{equation}\label{e:rec.f}
r(h,n,s) :=  \sum_{\substack{ m=0\\ (n-m)\ \mathrm{even}
}}^{n} c(h,m,s)\  r\!\left( h+1,\frac{n-m}{2},s+m2^h\right)
\end{equation}
with base cases
\begin{equation}\label{e:contrib.0}
c(h,0,s) =r(h,0,s)= 1.
\end{equation}

\bigskip

\begin{lemma}\label{lem:rec}
For $n\geq 1$,   the number of tanglegrams is $$t_n=\frac{r(0,n,0)}{f^2(n)},$$
which can be computed recursively using \eqref{e:rec.f}.
\end{lemma}

\begin{proof}
Let $\ttt_{n} := (1-2n)^{2} t_{n}$.  By the main
formula

\begin{equation}\label{e:rec}
\ttt_n = \sum_\lambda \frac{\prod_{i=1}^{\ell(\lambda)}
\left(2(\lambda_i+\cdots+\lambda_{\ell(\lambda)})-1\right)^2}{z_\lambda}.
\end{equation}
where the sum is over binary partitions of $n$.

We will consider the contribution to \eqref{e:rec} from the parts of
the partition of size $2^{h}$ for each $h$ separately.  To do this we
will need to keep track of the partial sums of parts smaller than
$2^{h}$.  Let $s^{\lambda} = (s^{\lambda}_{0},s^{\lambda}_{1},\ldots)$
where $s^\lambda_{h}=\sum_{i=0}^{h-1} m_{i}2^{i}$ and
$s^\lambda_0=0$.  Then the contribution of the parts of size $2^{h}$
in $\lambda$ to the corresponding term in \eqref{e:rec} is the factor
$c(h,m_{h},s^\lambda_{h})$.  Using this notation, we have
\begin{equation}
\ttt_n = \sum_{m^{\lambda}=(m_{0}, m_{1},\ldots) \vdash  n}c(0,m_{0},0) c(1,m_{1},s^\lambda_{1})c(2,m_{2},s^\lambda_{2})\cdots
\end{equation}
where the sum is over binary partitions of $n$ represented by their
multiplicity vector.

Next consider the binary partitions with exactly $j$ parts of size
1.  Note $n-j$ must be even for this set to be nonempty.  The binary
partitions of $n$ with exactly $j$ parts equal to 1 are in bijection
with the binary partitions of $\frac{n-j}{2}$,  so

\begin{equation}
\ttt_n = \sum_{\substack{ m_{0}=0\\ (n-m_{0})\ \mathrm{even}}}^{n}c(0,m_{0},0) \sum_{( m_{1},m_{2},\ldots) \vdash
  \frac{n-m_{0}}{2}}  c(1,m_{1},m_0)c(2,m_{2},m_0+2\cdot m_1)
\cdots .
\end{equation}
Observe that the recurrence in \eqref{e:rec.f} gives rise to the expansion
$$
r(h,n,s) = \sum_{(m_h, m_{h+1}, \ldots)\vdash n}
c(h,m_{h},s)
c(h+1,m_{h+1},s+m_h\cdot 2^h)
  c(h+2,m_{h+2},s+m_h\cdot 2^h+ m_{h+1}\cdot 2^{h+1})
  \cdots
  $$
\noindent
  where the sum is over binary partitions of $n$ but the indexing is
  shifted so $m_h$ is the number of parts of size 1.
Thus,
$$\ttt_n =   \sum_{\substack{ m=0\\ (n-m)\ \mathrm{even}
}}^{n} c(0,m,0)\ r\!\left(1,\frac{n-m}{2},m\right)= r(0,n,0)
$$
which completes the proof since $f^2(n)=(2n-1)^2$.
\end{proof}

\bigskip

We can extend the functions above to count tangled chains:
\begin{equation}\label{e:factor.k}
f^k(s) := (2s-1)^{k},
\end{equation}

\begin{equation}\label{e:contrib.k}
c^{k}(h,m,s) := \prod_{j=1}^{m} \frac{f^k(s + j\cdot 2^{h})}{j\cdot 2^{h}},
\end{equation}
and
\begin{equation}\label{e:rec.f.k}
r^{k}(h,n,s) :=  \sum_{\substack{ m=0\\ (n-m)\ \mathrm{even}
}}^{n} c^{k}(h,m,s)\  r\!\left( h+1,\frac{n-m}{2},s+m2^h\right)
\end{equation}
with base cases
\begin{equation}\label{e:contrib.0.k}
c^{k}(h,0,s) =r^{k}(h,0,s)= 1.
\end{equation}

Then a proof very similar to the case $k=2$ also proves the following statement.

\begin{cor}\label{cor:rec.k}
For $n\geq 1$,   the number of tangled chains of length $k$ is $$\frac{r^{k}(0,n,0)}{f^k(n)}$$
which can be computed recursively using \eqref{e:rec.f.k}.
\end{cor}

\section{Final remarks}\label{sec:remarks}

\subsection*{Generating functions}




It is known (and easy to prove) that the ordinary generating function for inequivalent trees satisfies the functional equation
$$B(x) = x + \frac 1 2 \left(B(x)^2 + B(x^2)\right).$$
This is, of course, equivalent to a recurrence for the sequence $b_n$. Given that in this paper we prove both explicit formulas and recurrences for the numbers of tanglegrams and tangled chains, it makes sense to ask the following.

\begin{question}
 Does there exist a closed form or a functional equation for the generating function of tanglegrams or tangled chains?
\end{question}

\subsection*{Number of cherries and other subtrees}

Cherries play an important role in the literature on tanglegrams.  For
example, Charleston's analysis \cite[pp. 325--326]{charleston} suggests
the following question.

\begin{question}
What is the expected number of matched cherries in a
random tanglegram?
\end{question}

Computer experiments with random tanglegram generation suggest that the following is true.

\begin{conj}
The expected number of cherries in the left tree in a random tanglegram converges to $n/4$.
\end{conj}

\begin{conj}
  The expected number of copies of the tree $T$ in the left tree of a random tanglegram of size $n$ is asymptotically equal to $2^{-(l+k-1)}n$, where $l$ is the number of leaves of $T$ and $k$ is the number of symmetries of $T$, i.e.\ vertices with identical subtrees.
\end{conj}


It also seems that the number of copies of a tree converges to a normal distribution.


\bigskip



If the conjectures hold, then for every tree $T$ with $l$ leaves and $k$ symmetries, the number of copies of the tree with $T$ as left and as right subtree in the left tree of a randomly chosen tanglegram asymptotically equals $2^{-(2l+(2k+1)-1)}n = 4^{-(l+k)} n$. So that would imply the following.

\begin{conj}
  Let $T' \in B_n$ be the left tree of a tanglegram chosen uniformly at random. The expected number of
  generators of $A(T')$ is asymptotically equal to
 $$\left( \sum_{T \in B_n} \frac 1{4^{l(T)+k(T)}} \right) n.$$
\end{conj}

It is not hard to see that the sum in the conjecture equals  $f({\textstyle \frac 1 4}) n$, where $f(x)$ is the function defined by $f(0) = 0$ and $f(x) = x + \frac
1 2 f(x)^2 + (x-\frac 1 2)f(x^2)$, or explicitly
 $$f(x) = 1 - \sqrt{1 - 2x + (1-2x)\left(1 - \sqrt{1 - 2x^2 + (1-2x^2)\left(1 - \sqrt{1-2x^4 + \cdots}\right)}\right)}.$$

Note that the computation of $f(\frac 1 4) = 0.27104169360883278703...$ converges very rapidly:
the number of correct digits roughly doubles after each step.




\medskip

\subsection*{Connection with symmetric functions}

The main theorems suggest that symmetric functions might be at play; note, for example, the similarity with the formula $h_n = \sum_\lambda z_\lambda^{-1} p_\lambda$, where $h_n$ is the homogeneous symmetric function, $p_\lambda$ the power sum symmetric function, and the sum is over all partitions of $n$.

\begin{question}
 Is there a connection between tanglegrams (or more generally tangled chains) and symmetric functions?
\end{question}

\begin{remark}
Based on a manuscript version of this paper, Ira Gessel pointed out
that there is indeed a connection between symmetric functions and the
enumeration of the ordered and unordered tanglegrams based on the
theory of species.  His claims will be spelled out in a forthcoming
paper \cite{gessel.2015}.
\end{remark}

\subsection*{Variants on tanglegrams}

Tanglegrams as described here fit in a set of more general setting of
pairs of graphs with a bijection between certain subsets of the
vertices (more completely described and motivated in
\cite{arniePaper}).  One can also consider \emph{unordered
  tanglegrams} by identifying $(T,v,S)$ with $(S,v^{-1}, T)$.  For
example, the 4th and 5th tanglegrams in Figure~\ref{fig:t4} are
equivalent as unordered tanglegrams, and so are the 8th and 10th.
From this picture, the reader can verify that there are 10 unordered
tanglegrams of size 4.

Because of reversibility assumptions for the continuous time Markov
mutation models commonly used to reconstruct phylogenetic trees,
unrooted trees are the most common output of phylogenetic inference
algorithms.  Thus another variant of tanglegrams involves using
unrooted trees in place of rooted ones.  The motivation for studying
these variants comes from noting that many problems in computational
phylogenetics such as distance calculation between trees
\cite{Allen2001-yz} ``factor'' through a problem on tanglegrams.

\begin{question}
Is there a nice formula for the number of
\begin{itemize}
\item unordered binary rooted tanglegrams,
\item ordered binary unrooted tanglegrams, or
\item unordered binary unrooted tanglegrams?
\end{itemize}
\end{question}
These counts have been found up to 9 leaves (Table~\ref{tab:counts}) by direct enumeration of double cosets \cite{arniePaper}.

\begin{table}[h]
\centering
\caption{The number of tanglegrams of various types up to 9 leaves.}
\label{tab:counts}
\begin{tabular}{lllll}
leaves & rooted ord. & rooted unord. & unrooted ord. & unrooted unord. \\
\hline
1      & 1             & 1            & 1               & 1              \\
2      & 1             & 1            & 1               & 1              \\
3      & 2             & 2            & 1               & 1              \\
4      & 13            & 10           & 2               & 2              \\
5      & 114           & 69           & 4               & 4              \\
6      & 1509          & 807          & 31              & 22             \\
7      & 25595         & 13048        & 243             & 145            \\
8      & 535753        & 269221       & 3532            & 1875           \\
9      & 13305590      & 6660455      & 62810           & 31929
\end{tabular}
\end{table}

\section*{Acknowledgments}

We thank Ira Gessel, Arnold Kas, Jim Pitman,  Xavier G. Viennot, Paul Viola, Bianca Viray, and Chris Whidden for helpful
discussions.

\bibliographystyle{siam}
\bibliography{the}

\end{document}